\documentclass[preprint, 3p, twocolumn]{elsarticle}

\usepackage{graphicx}
\usepackage{amsmath, amssymb, amsthm}
\usepackage{enumerate}
\usepackage{hyperref}
\hypersetup{colorlinks=true}
\usepackage{url}
\usepackage{color}
\usepackage{algorithm}
\usepackage[noend]{algorithmic}
\usepackage{inputenc}
\inputencoding{utf8}

\newcommand{\R}{\ensuremath{\mathbb{R}}}

\newcommand{\X}{\ensuremath{\mathbb{X}}}
\newcommand{\A}{\ensuremath{\mathbb{A}}}
\newcommand{\E}{\ensuremath{\mathbb{E}}}
\newcommand{\F}{\ensuremath{\mathbb{F}}}

\newcommand{\B}{\ensuremath{\mathcal{B}}}
\newcommand{\Gr}{\ensuremath{\mbox{\textnormal{Gr}}}}

\renewcommand{\H}{\ensuremath{\mathbb{H}}}
\renewcommand{\P}{\ensuremath{\mathbb{P}}}

\newtheorem{thm}{Theorem}
\newtheorem{prp}[thm]{Proposition}
\newtheorem{lem}[thm]{Lemma}
\newtheorem{cor}[thm]{Corollary}

\newtheorem{remark}[thm]{Remark}

\newtheorem{innercustomthm}{Assumption}
\newenvironment{customthm}[1]
  {\renewcommand\theinnercustomthm{#1}\innercustomthm}
  {\endinnercustomthm}

\begin{document}

\title{Reduction of total-cost and average-cost MDPs with weakly
  continuous transition probabilities to discounted MDPs}

\author[sbu]{Eugene A. Feinberg\corref{cor1}}
\ead{eugene.feinberg@stonybrook.edu}

\author[orie]{Jefferson Huang}
\ead{jefferson.huang@cornell.edu}

\cortext[cor1]{Corresponding author}

\address[sbu]{Department of Applied Mathematics and Statistics, Stony Brook University, Stony Brook, NY 11794-3600, USA}

\address[orie]{School of Operations Research and Information Engineering, Cornell University, Ithaca, NY 14853-3801, USA}

\begin{abstract}
  This note describes sufficient conditions under which total-cost and
  average-cost Markov decision processes (MDPs) with general state and
  action spaces, and with weakly continuous transition probabilities,
  can be reduced to discounted MDPs. For undiscounted problems, these reductions
  imply the validity of optimality equations and the existence of
  stationary optimal policies. The reductions also provide methods for
  computing optimal policies. The results are applied to a capacitated
  inventory control problem with fixed costs and lost sales.
\end{abstract}


\maketitle

\section{Introduction}
\label{sec:introduction}

Undiscounted Markov decision processes (MDPs) are typically much more
difficult to study than discounted MDPs. This is true both for models
with expected total costs and for models with average costs per unit
time.  This paper describes conditions under which undiscounted MDPs
with infinite state spaces and weakly continuous transition kernels
can be transformed into discounted MDPs.

For undiscounted total costs, a classic assumption is  
that the expected number of visits to each state in a certain set
$\X^\prime$ is finite under every policy and initial state. Such an
assumption is typically referred to as \emph{transience}
\cite[Chapter~7]{altman_constrained_1999},
\cite{veinott_discrete_1969}. When the expected amount of time spent
in $\X^\prime$ (i.e., the ``lifetime'' of the system) is finite for
every policy and initial state, the MDP is called \emph{absorbing}
\cite[Chapter~7]{altman_constrained_1999}. It is well-known that every
discounted MDP can be viewed as an absorbing MDP with the lifetime of
the system being geometrically distributed
\cite[p. 137]{altman_constrained_1999}. We remark that every absorbing
MDP is transient, and that the two conditions are equivalent when the
set $\X^\prime$ is finite.

For average costs per unit time, a classic approach has been to make
use of results about discounted MDPs. The most general results have
been obtained in \cite{feinberg_average_2012} using the so-called
vanishing discount factor approach, where the validity of optimality
inequalities and existence of stationary optimal policies are obtained
by considering optimality equations for discounted MDPs and letting
the discount factor tend to one. Another approach, which was used
early in the development of the theory of average-cost MDPs, is to
transform the average-cost problem into a discounted one, and argue
that optimal policies for the latter are also optimal for the former
\cite[Chapter 7 \S 10]{dynkin_controlled_1979},
\cite{ross_arbitrary_1968, ross_non-discounted_1968}. One advantage of
this approach is that it can be used to apply methods and algorithms
developed for discounted MDPs to undiscounted
MDPs. \cite{akian_policy_2013, feinberg_strong_2013,
  feinberg_reduction_2017}.

In \cite{feinberg_reduction_2017}, conditions were given under which
undiscounted MDPs with general state and action spaces can be reduced
to discounted ones. These conditions include the assumption that the
transition probabilities are setwise continuous. However, for many
models of interest, such as those arising in inventory control
\cite{feinberg_optimality_2016}, the transition probabilities are only
weakly continuous. In this paper, we provide conditions under which
the reductions in \cite{feinberg_reduction_2017} lead to optimality
results for undiscounted MDPs with weakly continuous transition
kernels. In particular, under these conditions the discounted MDPs to
which the undiscounted MDPs are reduced have weakly continuous
transition probabilities. Moreover, while sufficient conditions are
provided in \cite{costa_average_2012, feinberg_optimality_2017,
  jaskiewicz_optimality_2006} for the validity of the optimality
equations for average-cost MDPs, Assumption~\ref{asmp:ht} in
Section~\ref{sec:average-costs-per} ensures that a solution to this
optimality equation can be obtained via the optimality equation for a
discounted MDP. This in turn implies that such average-cost MDPs can
be solved using methods developed for discounted MDPs.

The rest of the paper is organized as follows. In
Section~\ref{sec:model-description}, the MDP model and objective
functions are described. Next, in Section~\ref{sec:total-costs} the
results for undiscounted total-cost MDPs are
presented. Section~\ref{sec:average-costs-per} contains the results
for average-cost MDPs. Finally, in Section~\ref{sec:example}
we apply the preceding results to a capacitated inventory control
problem with fixed ordering costs and lost sales.

\section{Model description}
\label{sec:model-description}

The \emph{state space} $\X$ and \emph{action space} $\A$ are Borel
subsets of complete separable metric spaces endowed with their
respective Borel $\sigma$-algebras $\B(\X)$ and $\B(\A)$. When the
current state is $x \in \X$, the decision-maker must select an action
from the \emph{set of available actions} $A(x)$, which is a nonempty
Borel subset of $\A$. The space of all feasible state-action pairs
\begin{displaymath}
  \Gr(A) := \{(x, a) | x \in \X, \ a \in A(x)\}
\end{displaymath}
is assumed to be a Borel subset of $\X \times \A$, and to contain the
graph of a Borel-measurable function from $\X$ to $\A$ (these
assumptions follow from Assumption \ref{asmp:w}(i) below). For each $(x, a) \in \Gr(A)$
there is an associated \emph{one-step cost} $c(x, a) \in [0, \infty)$
and a finite measure $q(\cdot | x, a)$ on $(\X, \B(\X))$. We assume that
the functions $(x, a) \mapsto c(x, a)$ and
$(x, a) \mapsto q(B | x, a)$, for each $B \in \B(\X)$, are
Borel-measurable. Moreover, $q$ is assumed to satisfy
\begin{displaymath}
  \sup\left\{ q(\X | x, a) : (x, a) \in \Gr(A)\right\} < \infty.
\end{displaymath}
For possible interpretations of the values $q(B | x, a)$ for
$B \in \B(\X)$, which may be greater than one, see
\cite[Section~2.1]{feinberg_reduction_2017}; in light of these
interpretations, we will refer to $q$ as the \emph{transition kernel}.

\subsection{Objective functions}
\label{sec:optimality-criteria}

Let $\H_0 := \X$, and for $n = 1, 2, \dots$ let
$\H_n := \X \times \A \times \H_{n-1}$ denote the space of all
\emph{histories} of the process up to decision epoch $n$, endowed with
the product $\sigma$-algebra. A \emph{decision rule} for epoch
$n = 0, 1, \dots$ is a mapping
$\pi_n:\B(\A) \times \H_n \rightarrow [0, 1]$ such that for every
$h_n = x_0 a_0 \cdots x_n$ the set function
$\pi_n(\cdot | h_n)$ is a probability measure on
$(\A, \B(\A))$ satisfying $\pi_n(A(x_n) | h_n) = 1$, and for every $B
\in \B(\A)$ the function $\pi_n(B | \cdot)$ on $\H_n$ is
Borel-measurable.

A \emph{policy} is a sequence $\pi = \{\pi_n\}_{n=0}^\infty$ of
decision rules; let $\Pi$ denote the set of all policies. Under a
policy $\pi$, at each decision epoch $n = 0, 1, \dots$ the
decision-maker observes the history
$h_n = x_0 a_0 \cdots x_n \in \H_n$ of the process up to epoch $n$ and
selects an action $a \in A(x_n)$ according to the probability
distribution $\pi_n(\cdot | h_n)$. A \emph{stationary policy} is
identified with a Borel-measurable function $\phi:\X \rightarrow \A$
satisfying $\phi(x) \in A(x)$ for all $x \in \X$; under such a policy,
the decision-maker selects the action $\phi(x)$ if the current state
is $x$. The set of all stationary policies is denoted by $\F$.


To define the objective functions under consideration, for
$B \in \B(\X)$ and $(x, a) \in \Gr(A)$ let
\begin{displaymath}
  p(B | x, a) := q(B | x, a) / q(\X | x, a),
\end{displaymath}
and let
\begin{displaymath}
\alpha(x, a) := q(\X | x, a).
\end{displaymath}
Observe that $p(\cdot | x, a)$
is a probability measure on $(\X, \B(\X))$ for every
$(x, a) \in \Gr(A)$, and that $p(B | \cdot)$ is a Borel function on
$\Gr(A)$ for every $B \in \B(\X)$. Therefore, for every policy
$\pi \in \Pi$ and initial state $x \in \X$ the Ionescu Tulcea theorem
\cite[pp. 140-141]{bertsekas_stochastic_1978} uniquely defines a
probability measure $\P_x^\pi$ on
$((\X \times \A)^\infty, \B[(\X \times\A)^\infty])$ and its associated
expectation operator $\E_x^\pi$.

When the initial state is $x \in \X$, under $\pi \in \Pi$ the
\emph{total cost} incurred is
\begin{displaymath}
  v^\pi(x) := \E_x^\pi\sum_{n=0}^\infty \alpha(x_n, a_n)c(x_n, a_n),
\end{displaymath}
and the \emph{average cost} incurred is
\begin{displaymath}
  w^\pi(x) :=
  \limsup_{N\rightarrow\infty}\frac{1}{N}\E_x^\pi\sum_{n=0}^{N-1}
  \alpha(x_n, a_n)c(x_n, a_n).
\end{displaymath}
A policy $\pi_* \in \Pi$ is \emph{total-cost optimal} if
\begin{displaymath}
  v^{\pi_*}(x) = \inf_{\pi \in \Pi}v^\pi(x) =: v(x) \qquad \forall x
  \in \X,
\end{displaymath}
and is \emph{average-cost optimal} if
\begin{displaymath}
  w^{\pi_*}(x) = \inf_{\pi \in \Pi}w^\pi(x) =: w(x) \qquad \forall x
  \in \X.
\end{displaymath}
If there exists a constant $\beta$ such that $\alpha(x, a) = \beta$ for all
$(x, a) \in \Gr(A)$, a total-cost optimal policy is called \emph{$\beta$-optimal}.

\section{Total costs}
\label{sec:total-costs}

To state Assumption~\ref{asmp:t} below for the total-cost criterion, given $\phi \in
 \F$ and a Borel function $u:\X \rightarrow \R$ let
\begin{displaymath}
  Q_\phi u(x) := \int_\X u(y)q(dy | x, \phi(x)), \qquad x \in \X,
\end{displaymath}
let $Q_\phi^0 u(x) := u(x)$ for $x \in \X$, and for $n = 1, 2, \dots$ let
$Q_\phi^n u(x) := Q_\phi(Q_\phi^{n-1}u)(x)$ for $x \in \X$.
\begin{customthm}{T}\label{asmp:t}
  There exists a continuous function $V:\X \rightarrow [1, \infty)$
  and a constant $K$ satisfying
  \begin{equation}
    \label{eq:asmp-t}
    \sum_{n=0}^\infty Q_\phi^nV(x) \leq KV(x) < \infty, \ \ \forall
    \phi \in \F, \ x \in \X.
  \end{equation}
\end{customthm}

The statement of Assumption~\ref{asmp:w} below requires several
definitions. Let $\mathbb{S}$ and $\mathbb{T}$ be metric spaces
endowed with their respective Borel $\sigma$-algebras $\B(\mathbb{S})$
and $\B(\mathbb{T})$. A set-valued mapping
$s \mapsto \Phi(s) \subseteq \mathbb{T}$ on $\mathbb{S}$ is
\emph{compact-valued} if $\Phi(s)$ is compact for all
$s \in \mathbb{S}$, and is \emph{continuous on} $\mathbb{S}$ if
for every open set $V \subseteq \mathbb{T}$ the sets
$\{s \in \mathbb{S} | \Phi(s) \subseteq V\}$ and
$\{s \in \mathbb{S} | \Phi(s) \cap V \neq \emptyset\}$ are open in
$\mathbb{S}$.

Next, a \emph{transition kernel} from $\mathbb{S}$ to $\mathbb{T}$ is
a mapping
$\kappa: \B(\mathbb{T}) \times \mathbb{S} \rightarrow [0, \infty)$
such that $\kappa(\cdot | s)$ is a finite measure on
$(\mathbb{T}, \B(\mathbb{T}))$ for every $s \in \mathbb{S}$, and
$\kappa(T | \cdot)$ is a Borel function on $\mathbb{S}$ for every
$T \in \B(\mathbb{T})$. A transition kernel $\kappa$ is \emph{weakly
  continuous} if for every bounded continuous function $f:\mathbb{T}
\rightarrow \R$ the mapping
\begin{displaymath}
  s \mapsto \int_{\mathbb{T}}f(t)\kappa(dt | s)
\end{displaymath}
is continuous on $\mathbb{S}$.
If $\kappa$ is a transition kernel such that
$\kappa(\cdot | s)$ is a probability measure for every
$s \in \mathbb{S}$, it is called a \emph{transition probability
  kernel}.

Finally, a function $f:\mathbb{S} \rightarrow \R$ is \emph{lower
  semicontinuous at} $s \in \mathbb{S}$ if
$\liminf_{s^\prime \rightarrow s}f(s^\prime) \geq f(s)$, and is
\emph{lower semicontinuous on} $S \subseteq \mathbb{S}$ if it is lower
semicontinuous at every $s \in
S$. 


\begin{samepage}
  \begin{customthm}{WC}\label{asmp:w} \
    \begin{enumerate}[(i)]
    \item The set-valued mapping $x \mapsto A(x)$ is compact-valued
      and continuous on $\X$.
    \item The transition kernel $q$ is weakly continuous.
    \item The function $(x, a) \mapsto c(x, a)$ is lower
      semicontinuous on $\Gr(A)$.
    \end{enumerate}
  \end{customthm}
\end{samepage}

\begin{prp}\label{prp:sh}
  Suppose Assumptions~\ref{asmp:t} and \ref{asmp:w}(i, ii)
  hold. Then there exists a continuous function
  $\mu:\X \rightarrow [1, \infty)$ satisfying
  $V(x) \leq \mu(x) \leq KV(x)$ for all $x \in \X$ and
  \begin{equation}
    \label{eq:sh}
    \mu(x) \geq V(x) + \int_\X \mu(y) q(dy | x, a)
  \end{equation}
  for all $(x, a) \in \Gr(A)$.
\end{prp}
\begin{proof}
  Consider the operator $\mathcal{U}$ defined for Borel functions
  $u:\X \rightarrow \R$ by
  \begin{displaymath}
    \mathcal{U}u(x) := \sup_{ a\in A(x)}\left[V(x) + \int_\X u(y)q(dy |
      x, a)\right]
  \end{displaymath}
  for $x \in \X$. Let $u_0 \equiv 0$, and for $n = 1, 2, \dots$ let
  $u_n := \mathcal{U}u_{n-1}$. According to the Berge maximum theorem
  (see e.g., \cite[p. 570]{aliprantis_infinite_2006}), for
  $n = 0, 1, \dots$ the function $u_n$ is continuous. Since
  $u_{n+1} \geq u_n \geq V$ pointwise for $n = 1, 2, \dots$, the
  sequence of continuous functions $\{u_n\}_{n=0}^\infty$ converges to
  a Borel function $\mu := \lim_{n\rightarrow\infty}u_n \geq V$. The
  claim that $\mu \leq K\mu$ can be verified using the arguments in
  \cite[Proof of Proposition~1]{feinberg_reduction_2017} and the Berge
  maximum theorem. Moreover, Lebesgue's monotone convergence theorem
  implies that $\mu = \mathcal{U}\mu$, which means (\ref{eq:sh}) holds
  for all $(x, a) \in \Gr(A)$.

  It remains to be shown that the function $\mu:\X \rightarrow \R$
  defined above is continuous. 
  First, observe that for any Borel functions
  $f, g$ on $\X$,
  \begin{displaymath}
    f(x) \leq g(x) + \mu(x)\left(\sup_{x \in \X}\frac{|f(x) -
      g(x)|}{\mu(x)}\right), \ \forall x \in \X,
  \end{displaymath}
  which implies that for all $x \in \X$,
  \begin{displaymath}\footnotesize
    \begin{aligned}
      \mathcal{U}f(x) &\leq \mathcal{U}g(x) + (\mu(x) -
      V(x))\left(\sup_{x \in \X}\frac{|f(x) - g(x)|}{\mu(x)}\right) \\
      &\leq \mathcal{U}g(x) + \mu(x)\left(\frac{K-1}{K}\right) \left(\sup_{x \in \X}\frac{|f(x) -
        g(x)|}{\mu(x)}\right).
    \end{aligned}
  \end{displaymath}
  By reversing the roles of $f$ and $g$, it follows that
  \begin{displaymath}\footnotesize
    \begin{aligned}
      \frac{|\mathcal{U}f(x) - \mathcal{U}g(x)|}{\mu(x)} \leq
      \left(\frac{K-1}{K}\right)\sup_{x \in \X}\frac{|f(x) -
        g(x)|}{\mu(x)}, \ \forall x \in \X.
    \end{aligned}
  \end{displaymath}
  Since $V \leq \mu \leq KV$, it follows that for the sequence
  $\{u_n\}_{n=0}^\infty$ defined above,
  \begin{displaymath}\small
    \begin{aligned}
      \sup_{x \in \X}\frac{|u_{n+1}(x) - u_n(x)|}{KV(x)} \leq
      \left(\frac{K-1}{K}\right)^n, \ n = 0, 1, \dots,
    \end{aligned}
  \end{displaymath}
  which implies that for all nonnegative integers $m,n$ satisfying
  $m > n$,
  \begin{align}
      \sup_{x \in \X}&\frac{|u_m(x) - u_n(x)|}{KV(x)} \notag\\
      &\leq \sum_{k=0}^{m-n-1}\sup_{x \in
          \X}\frac{|u_{n+k+1}(x) - u_{n+k}(x)|}{KV(x)} \notag\\
      &\leq
      \sum_{k=0}^{m-n-1}\left(\frac{K-1}{K}\right)^{n+k} \notag\\
      &\leq
      \left(\frac{K-1}{K}\right)^n\sum_{k=0}^\infty\left(\frac{K-1}{K}\right)^k\notag\\
      &= K\left(\frac{K-1}{K}\right)^n.\label{eq:cauchy}
  \end{align}
  Define the \emph{$V$-norm} for functions $f:\X \rightarrow \R$ by
  $\|f\|_V := \sup_{x \in \X}|f(x)|/V(x)$, and let $C_V(\X)$ denote
  the space of continuous functions on $\X$ with finite $V$-norm. Then
  (\ref{eq:cauchy}) implies that $\{u_n\}_{n = 0}^\infty$ is a Cauchy
  sequence in $C_V(\X)$. Since $C_V(\X)$ is a Banach space with
  respect to $\|\cdot\|_V$, it follows that the sequence
  $\{u_n\}_{n=0}^\infty$ converges to a function in $C_V$. Since
  $\lim_{n\rightarrow\infty}u_n = \mu$, it follows that $\mu \in C_V$;
  in particular, $\mu$ is continuous.
\end{proof}

\subsection{Hoffman-Veinott (HV) transformation}
\label{sec:hoffm-vein-transf}

In this section, we present the HV transformation
\cite{feinberg_reduction_2017}, which is based on ideas due to Alan
Hoffman and A. F. Veinott \cite{veinott_discrete_1969}. A point $s$ is
\emph{isolated} from a metric space $\mathbb{S},$ if there exists an
$\epsilon > 0$ such that the distance between $s$ and any element of
$\mathbb{S}$ is larger than $\epsilon$
. The state space of the new MDP is
$\tilde{\X} := \X \cup \{\tilde{x}\}$, where $\tilde{x} \not\in \X$ is
a cost-free absorbing state that is isolated from $\X$. The
action space is $\tilde{\A} := \A \cup \{\tilde{a}\}$, where
$\tilde{a}$ is the only action available when the current state is
$\tilde{x}$. The set $\tilde{A}(x)$ of available actions is unchanged
if the current state $x$ is not $\tilde{x}$, i.e.,
\begin{displaymath}
  \tilde{A}(x) :=
  \begin{cases}
    A(x), &\quad \text{if} \ x \in \X, \\
    \{\tilde{a}\}, &\quad \text{if} \ x = \tilde{x}.
  \end{cases}
\end{displaymath}
The one-step cost function $\tilde{c}$ is defined by
\begin{displaymath}
  \tilde{c}(x, a) :=
  \begin{cases}
    \mu(x)^{-1}c(x, a), &\quad \text{if} \ (x, a) \in \Gr(A), \\
    0, &\quad \text{if} \ (x, a) = (\tilde{x}, \tilde{a}).
  \end{cases}
\end{displaymath}
Finally, select a discount factor
\begin{displaymath}
  \tilde{\beta} \in [(K-1)/K, 1),
\end{displaymath}
and define the transition probabilities $\tilde{p}$ as follows. For
$(x, a) \in \Gr(A)$, let
\begin{displaymath}
  \tilde{p}(B | x, a) := \frac{1}{\tilde{\beta}\mu(x)}\int_B
  \mu(y)q(dy | x, a), \ \ B \in \B(\X),
\end{displaymath}
\begin{displaymath}
  \tilde{p}(\{\tilde{x}\} | x, a) := 1 -
  \frac{1}{\tilde{\beta}\mu(x)}\int_\X \mu(y)q(dy | x, a),
\end{displaymath}
and let
\begin{displaymath}
  \tilde{p}(\{\tilde{x}\} | \tilde{x}, \tilde{a}) := 1.
\end{displaymath}
Since only one action is available in state $\tilde{x}$, and the
action sets coincide otherwise, there is a one-to-one correspondence
between policies for the new MDP and the original MDP.

For $x\in\tilde\X$ and $\pi \in \Pi$, let $\tilde v^\pi(x),$ be the
expected total discounted cost for the new model, and let
$\tilde v(x) := \inf_{\pi \in \Pi}\tilde{v}^\pi(x)$.  It is well-known
(see e.g., \cite{feinberg_reduction_2017}) that
$\tilde v^\pi(x)=\mu(x)^{-1}v^\pi(x)$ and
$\tilde v(x)=\mu(x)^{-1}v(x)$ for all $x\in\X.$ 

\begin{thm}\label{thm:hv-weak}
  Suppose Assumptions~\ref{asmp:t} and 
  \ref{asmp:w}(i,ii) hold. If the function
  \begin{equation}\label{eq:cont-V}
    (x, a) \mapsto \int_\X V(y)q(dy | x, a)
  \end{equation}
  is continuous on $\Gr(A)$, then $\tilde{p}$ is a weakly continuous
  transition probability kernel. In addition, if
  Assumption~\ref{asmp:w}(iii) holds, then there exists a stationary
  $\tilde{\beta}$-optimal policy for the MDP obtained from the HV
  transformation, and for this MDP a stationary policy is
  $\tilde{\beta}$-optimal if and only if for all $x \in \X$,
  \begin{equation}\label{eq:3}
  \begin{aligned} {\tilde v}(x) &= {\tilde c}(x, \phi(x))+{\tilde
      \beta}\int_\X{\tilde v}(y){\tilde p}(dy|x,\phi(x)) \\
    &=\min_{a\in
      A(x)}\left[ \tilde{c}(x, a) + \tilde{\beta}\int_\X
      \tilde{v}(y)\tilde{p}(dy | x, a) \right].
 \end{aligned}
\end{equation}
\end{thm}
\begin{proof}
  According to Proposition~\ref{prp:sh}, the function $\mu$ used in
  the HV transformation can be taken to be continuous. Moreover,
  Assumption~\ref{asmp:t} implies that the function $V$ is integrable
  with respect to $q(\cdot | x, a)$, for all $(x, a) \in
  \Gr(A)$. Since $\mu \leq KV$, the weak continuity of $\tilde{p}$
  then follows from Lemma~\ref{lem:dct} in the Appendix.   

  Next, recalling that $\tilde{x}$ is isolated from $\X$, the
  continuity of $\mu$ by Proposition~\ref{prp:sh} implies that the
  nonnegative function $\tilde{c}$ is lower semicontinuous on
  $\Gr(\tilde{A})$. Since the action sets $\tilde{A}(x)$ are compact
  for all $x \in \tilde{\X}$, it follows from
  \cite[Theorem~2]{feinberg_average_2012} that the value function
  $\tilde{v}$ for the discounted MDP defined by the HV transformation
  satisfies
  \begin{displaymath}
    \label{eq:hv-dcoe}
    \tilde{v}(x) = \min_{a \in \tilde{A}(x)}\left[ \tilde{c}(x, a) +
      \tilde{\beta}\int_{\tilde{\X}}\tilde{v}(y)\tilde{p}(dy | x, a)\right]
  \end{displaymath}
  for all $x \in \tilde{\X}$, and there exists a stationary optimal
  policy for this discounted problem. Moreover, since
  $\tilde{v}(\tilde{x}) = 0$, a stationary policy $\phi$ is optimal
  for the discounted problem if and only if (\ref{eq:3}) holds for all
  $x \in \tilde{\X}$. The need to only consider $x \in \X$, for which
  $A(x) = \tilde{A}(x)$, follows from the fact that there is only one
  available action at state $\tilde{x}$.
\end{proof}


\begin{cor}\label{cor:hv-weak}
  Suppose Assumptions~\ref{asmp:t} and \ref{asmp:w} hold and that the
  mapping (\ref{eq:cont-V}) on $\Gr(A)$ is continuous. Then
  \begin{enumerate}[(i)]
  \item the value function $v$ satisfies the optimality equation
    \begin{displaymath}\label{eq:tcoe}\small
      v(x) = \min_{a\in A(x)}\left[c(x, a) + \int_\X v(y)q(dy | x, a)\right]
    \end{displaymath}
    for all $x \in \X$;
  \item there exists a stationary policy that is total-cost optimal;
  \item a stationary policy $\phi$ is total-cost optimal if and only
    if
    \begin{displaymath}
      \label{eq:eq:tcop}
      v(x) = c_\phi(x) + Q_\phi v(x) \quad \forall x \in \X,
    \end{displaymath}
    which holds if and only if $\phi$ is $\tilde{\beta}$-optimal for
    the MDP defined by the HV transformation.
  \end{enumerate}
\end{cor}
\begin{proof}
  This follows from Theorem~\ref{thm:hv-weak}, the definition of the
  HV transformation, and the fact that $v(x) = \mu(x) \tilde{v}(x)$
  for all $x \in \X$.
\end{proof}

\section{Average costs per unit time}
\label{sec:average-costs-per}

To state Assumption~\ref{asmp:ht} below, given $\phi \in \F$, a Borel
function $u:\X \rightarrow \R$, and a
state $z \in \X$, let
\begin{displaymath}
  \mathbin{_z Q_\phi}u(x) := \int_{\X \setminus \{z\}}u(y)q(dy | x,
  a), \qquad x \in \X,
\end{displaymath}
define $\mathbin{_z Q_\phi^0}u(x) \equiv u(x)$ for $x \in \X$, and for
$x \in \X$ and $n = 1, 2, \dots$ let
$\mathbin{_z Q_\phi^n}u(x) := \mathbin{_x Q_\phi}(\mathbin{_x
  Q_\phi^{n-1}}u)(x)$. Also, let $\mathbf{e}(x) := 1$ for $x \in \X$.

\begin{customthm}{HT}\label{asmp:ht}
  There exists a state $\ell \in \X$ and a constant $K_\ell$
  satisfying
  \begin{equation}
    \label{eq:asmp-ht}
    \sum_{n=0}^\infty\mathbin{_\ell Q_\phi^n}\mathbf{e}(x) \leq K_\ell
    < \infty, \ \ \forall \phi \in \F, \ x \in \X.
  \end{equation}
\end{customthm}
\begin{prp}\label{prp:sh-ht}
  Suppose Assumption~\ref{asmp:ht} holds with a state $\ell \in \X$
  that is isolated from $\X$, and 
  Assumptions~\ref{asmp:w}(i,ii) hold. Then there exists a continuous
  function $\mu_\ell:\X \rightarrow [1, \infty)$ satisfying
  $\mu_\ell(x) \leq K_\ell$ for all $x \in \X$ and
  \begin{equation}
    \label{eq:sh-ht}
    \mu_\ell(x) \geq 1 + \int_{\X \setminus \{\ell\}}\mu_\ell(y)q(dy |
    x, a)
  \end{equation}
  for all $(x, a) \in \Gr(A)$.
\end{prp}
\begin{proof}
  Consider the transition kernel $q_\ell$ from
  $\Gr_\ell(A) := \{(x, a) \in \Gr(A) | x \neq \ell\}$ to $\X_\ell := \X \setminus \{\ell\}$ where
  \begin{displaymath}
    q_\ell(B | x, a) := q(B \setminus \{\ell\} | x, a)
  \end{displaymath}
  for $B \in \B(\X_\ell)$ and $(x, a) \in \Gr_\ell(A)$. Then it
  follows from Proposition~\ref{prp:sh} and Assumption~\ref{asmp:ht}
  that there exists a continuous function
  $\mu_\ell:\X_\ell \rightarrow [1, \infty)$ that is bounded above by
  \begin{displaymath}
  K_\ell^{-} := \sup_{x \in \X \setminus \{\ell\}}\left\{\sum_{n=0}^\infty\mathbin{_\ell
    Q_\phi^n}\mathbf{e}(x)\right\}
\end{displaymath}
and satisfies (\ref{eq:sh-ht}) for all
  $(x, a) \in \Gr_\ell(A)$. Letting
  \begin{displaymath}
    \mu_\ell(\ell) := \sup_{A(\ell)}\left[1 + \int_{\X \setminus \{\ell\}}\mu_\ell(y)q(dy
    | x, a)\right]
  \end{displaymath}
  and recalling that $\ell$ is isolated from $\X$, it follows that
  this extension of $\mu_\ell$ to $\X$ is continuous and bounded above
  by $K_\ell$ according to Assumption~\ref{asmp:ht}, and satisfies
  (\ref{eq:sh-ht}) for all $(x, a) \in \Gr(A)$.
\end{proof}

\begin{remark}\label{rem:1}
  The function $\mu_\ell$ that is constructed in the proof of
  Proposition~\ref{prp:sh-ht} gives, for each $x \in \X$, the supremum
  $\mu_\ell(x)$ (over all policies) of the expected number of epochs
  before the system hits state $\ell$ after epoch 1. If the state
  $\ell$ is not isolated, then this function $\mu_\ell$ may be
  discontinuous at $\ell$ despite the weak continuity of $q$.

To verify this, let
$\ell := (\sqrt{5} - 1)/2$ and consider the following MDP with only
one available action $a_0$ for each state and a constant one-step cost
function.  The state space is the closed interval $\X := [0, \ell]$,
and the transition probabilities $q(\cdot | x, a_0)$ are defined for
$x \in \X$ as follows. Let $q(\{\ell\} | 0, a_0) := 1$,
$q(\{\ell\} | \ell, a_0) := 1 - \ell$, and
$q(\{0\} | \ell, a_0) := \ell$. In addition, for $x \in (0, \ell)$ let
$q(\{x\} | x, a_0) := x^2$, $q(\{\ell\} | x, a_0) := 1 - x - x^2$, and
$q(\{0\} | x, a_0) := x$. 

  Observe that Assumption~\ref{asmp:ht} holds because
  $\mu_\ell(0) = 1$, $\mu_\ell(\ell) = (\sqrt{5} + 1)/2$, and
  $\mu_\ell(x) = 1/(1 - x) \leq (\sqrt{5}+3)/2$ for $x \in (0,
  \ell)$. Moreover, it is straightforward to verify that this MDP
  satisfies Assumptions~\ref{asmp:w}(i,ii). On the other
  hand, since
  $\lim_{x\rightarrow\ell}\mu_\ell(x) = 1/(1 - \ell) = (\sqrt{5} +
  3)/2 > (\sqrt{5} + 1)/2 = \mu_\ell(\ell)$, the function $\mu_\ell$
  is discountinuous at $\ell$.
\end{remark}

\subsection{HV-AG transformation}
\label{sec:hv-ag-transformation}

Suppose Assumption~\ref{asmp:ht} holds. We now describe the
\emph{HV-AG transformation} \cite{feinberg_reduction_2017}, which is
based on the work of Akian \& Gaubert \cite{akian_policy_2013}. As was
the case with the HV tranformation, the HV-AG transformation results
in a discounted MDP
, whose set of
policies corresponds to the set of policies for the original MDP. 

The components of the discounted MDP defined by the HV-AG
transformation will be indicated by a horizontal bar. The state space
is $\bar{\X} := \X \cup \{\bar{x}\}$, where $\bar{x} \not\in \X$ is a
cost-free absorbing state that is isolated from $\X$. The action space
is $\bar{\A} := \A \cup \{\bar{a}\}$, where $\bar{a}$ is the only
action available when the system is in state $\bar{x}$. The set
$\bar{A}(x)$ of available actions is unchanged if the current state
$x$ is not $\bar{x}$, i.e.,
\begin{displaymath}
  \bar{A}(x) :=
  \begin{cases}
    A(x), &\quad \text{if} \ x \in \X, \\
    \{\bar{a}\}, &\quad \text{if} \ x = \bar{x}.
  \end{cases}
\end{displaymath}
The one-step cost function $\bar{c}$ is defined by
\begin{displaymath}
  \bar{c}(x, a) :=
  \begin{cases}
    \mu_\ell(x)^{-1}c(x, a), &\quad \text{if} \ (x, a) \in \Gr(A), \\
    0, &\quad \text{if} \ (x, a) = (\bar{x}, \bar{a}).
  \end{cases}
\end{displaymath}
Finally, select a discount factor
\begin{displaymath}
  \bar{\beta} \in [(K_\ell - 1)/K_\ell, 1),
\end{displaymath}
and define the transition probabilities $\bar{p}$ as follows. For $(x,
a) \in \Gr(A)$ and $B \in \B(\X \setminus \{\ell\})$, let
\begin{displaymath}
\bar{p}(B | x, a) :=
  \frac{1}{\bar{\beta}\mu_\ell(x)}\int_B\mu_\ell(y)q(dy | x, a),
\end{displaymath}
and let
\begin{displaymath}
  \bar{p}(\{\ell\} | x, a) := \frac{\mu_\ell(x) - 1 - \int_{\X
  \setminus \{\ell\}}\mu_\ell(y) q(dy | x, a)}{\bar{\beta}\mu_\ell(x)},
\end{displaymath}
\begin{displaymath}
  \bar{p}(\{\bar{x}\} | x, a) := 1 - \frac{\mu_\ell(x) - 1}{\bar{\beta}\mu_\ell(x)}.
\end{displaymath}
Finally, let
\begin{displaymath}
  \bar{p}(\{\bar{x}\} | \bar{x}, \bar{a}) := 1.
\end{displaymath}
Since only the action $\bar{a}$ is available at the state $\bar{x}$ and the
action sets coincide otherwise, there is a one-to-one correspondence
between policies for the new MDP and the original MDP.

For $x\in\bar\X$ and $\pi \in \Pi$, let $\bar v^\pi(x),$ be the
expected total discounted cost for the new model, and let
$\bar v(x) := \inf_{\pi \in \Pi}\bar{v}^\pi(x)$.

\begin{thm}\label{thm:hv-ag-weak}
  Suppose Assumption~\ref{asmp:ht} holds with a state $\ell \in \X$
  that is isolated from $\X$, and 
  Assumptions~\ref{asmp:w}(i,ii) hold. Then $\bar{p}$ is a weakly
  continuous transition probability kernel. In addition, if
  Assumption~\ref{asmp:w}(iii) holds, then there exists a stationary
  $\bar{\beta}$-optimal policy for the MDP defined by the HV-AG
  transformation, and for this MDP a stationary policy is
  $\bar{\beta}$-optimal if and only if for all $x \in \X$,
  \begin{equation}
    \label{eq:4}
    \begin{aligned}
      \bar{v}(x) &= \bar{c}(x, \phi(x)) +
      \bar{\beta}\int_\X\bar{v}(y)\bar{p}(dy | x, \phi(x)) \\
      &= \min_{a \in A(x)}\left[ \bar{c}(x, a) +
        \bar{\beta}\int_\X\bar{v}(y)\bar{p}(dy | x, a)\right].
    \end{aligned}
  \end{equation}

\end{thm}
\begin{proof}
  Proposition~\ref{prp:sh-ht} implies that the function $\mu_\ell$
  used in the HV-AG transformation can be taken to be
  continuous. Since $\mu_\ell \leq K_\ell < \infty$, the weak
  continuity of $\bar{p}$ follows from Lemma~\ref{lem:dct} in the
  Appendix.

  Next, observe that $\bar{c}$ is lower semicontinuous on
  $\Gr(\bar{A})$, and the action sets $\bar{A}(x)$ are compact for all
  $x \in \bar{\X}$. According to
  \cite[Theorem~2]{feinberg_average_2012}, it follows that
  \begin{displaymath}
    \bar{v}(x) = \min_{ a\in \bar{A}(x)}\left[\bar{c}(x, a) +
      \bar{\beta}\int_{\bar{\X}}\bar{v}(y)\bar{p}(dy | x, a)\right]
  \end{displaymath}
  for all $x \in \bar{\X}$, there exists a stationary optimal
  policy for the discounted problem, and a stationary policy is
  optimal for this problem if and only if (\ref{eq:4}) holds for all
  $x \in \X$.
\end{proof}
\begin{cor}\label{cor:hv-ag-weak}
  Suppose Assumption~\ref{asmp:ht} holds with a state $\ell \in \X$
  that is isolated from $\X$ and Assumption~\ref{asmp:w} holds. Then
  \begin{enumerate}[(i)]
  \item the constant $w := \bar{v}(\ell)$ and the function
    $h(x) := \mu(x)[\bar{v}(x) - \bar{v}(\ell)]$, $x \in \X$, satisfy
    \begin{displaymath}\small
      w + h(x) = \min_{ a\in A(x)}\left[c(x, a) + \int_\X h(y)q(dy |
        x, a)\right]
    \end{displaymath}
    for all $x \in \X$, and
  \item if the one-step cost function $c$ is bounded, and $q$ is a
    transition probability kernel, then there exists a stationary
    average-cost optimal policy, and any stationary policy $\phi$
    satisfying
    \begin{displaymath}
      w + h(x) = c_\phi(x) + Q_\phi h(x) \quad \forall x \in \X,
    \end{displaymath}
    where $w$ are $h$ are defined in (i), is average-cost optimal;
  \item there exists a $\bar{\beta}$-optimal stationary policy for the
    MDP defined by the HV-AG transformation, and under the hypotheses
    of (ii) every such policy is
    average-cost optimal for the original MDP.
  \end{enumerate}
\end{cor}

\begin{proof}
  Statement (i) follows from Theorem~\ref{thm:hv-ag-weak} and the
  definition of the HV-AG transformation. Moreover, statement (ii)
  follows from statement (i) and
  \cite[Proposition~5.5.5]{hernandez-lerma_discrete-time_1995}. Finally,
  statement (iii) follows from Theorem~\ref{thm:hv-ag-weak}, statement (ii), the definition of the
  HV-AG transformation.
\end{proof}


\section{Capacitated inventory control with fixed ordering costs and lost
  sales}
\label{sec:example}

Consider the following single-item \emph{capacitated}
\emph{periodic-review} inventory control problem with \emph{fixed
  ordering costs} and \emph{lost sales}. At each period
$n = 0, 1, \dots$, the decision-maker observes the current
\emph{inventory level} $x_n$ and places an \emph{order} $a_n \geq
0$. After the order is received in the same period, the demand
$D_{n+1} \geq 0$ is realized. Any remaining inventory is held to the
next period, and all unmet demand is lost. The demands
$D_1, D_2, \dots$ are assumed to be independent and identically
distributed with distribution $G_D(\cdot)$, where $G_D(0) <
1$. Moreover, we assume that the system is \emph{capacitated}, where
the inventory level can be at most $C < \infty$ and the maximum order
size is $M < \infty$.

Whenever a positive amount is ordered, a fixed
cost $K \geq 0$ is incurred in addition to a per-unit cost of
$\overline{c} > 0$. The cost to hold $x$ units of inventory for one
period is $h(x)$, where $h:[0, C] \rightarrow [0, \infty)$ is
assumed to be continuous.

The inventory control problem described above can be formulated as an
MDP as follows. The state space is $\X := [0, C] \cup \{0_L\}$, where
$0_L$ is isolated from $[0, C]$. The special state $0_L$, which
indicates the occurrence of a lost sale, will be used to apply the
results in Section~\ref{sec:average-costs-per}. For every $x \in \X$,
the set of available actions is $A(x) \equiv \A := [0, M]$.


Letting $0_L + y := y$ for $y \in \R$, the state process can be
described by the stochastic equation
\begin{equation}\footnotesize
  \begin{aligned}
    x_{n+1} &= F(x_n, a_n, D_{n+1}) \\
    &:=
    \begin{cases}
      \min\{x_n + a_n - D_{n+1}, C\}, & x_n + a_n \geq D_{n+1}, \\
      0_L, & x_n + a_n < D_{n+1}.
    \end{cases}
  \end{aligned}\notag
\end{equation}
This equation defines the transition probability kernel $q$ for the
corresponding MDP, where
\begin{displaymath}
  q(B | x, a) := \int_B \mathbf{1}\{F(x, a, s) \in B\}~dG_D(s)
\end{displaymath}
for $B \in \B(\X)$ and $(x, a) \in \X \times \A$, where
$\mathbf{1}\{\cdot\}$ denotes the indicator function. Since $0_L$ is
isolated from $\X$ and $F$ is continuous on
$\X \times \A \times [0, \infty)$, it follows that $q$ is weakly
continuous; see e.g., \cite[p.~ 92]{hernandez-lerma_adaptive_1989}.

Recall that $K \geq 0$ is the \emph{fixed ordering cost},
$\overline{c} \geq 0$ is the \emph{per-unit ordering cost}, and
$h:\X \rightarrow [0, \infty)$ is the per-period \emph{holding cost
  function}. It follows that the associated one-step cost function
$c:\X \times \A \rightarrow [0, \infty)$ is given by $c(x, a) :=
  K\mathbf{1}\{a > 0\} + \overline{c}a + \int_0^\infty
  h[F(x, a, s)]~dG_D(s)$.
  Since $h$ is continuous on $[0, C]$, $c$ is bounded on
  $\X \times \A$. Moreover, for every $\lambda \in \R$, the set
  $\{(x, a) \in \X \times \A | c(x, a) \leq \lambda\}$ is a compact
  subset of $\X \times \A$; this implies that $c$ is lower
  semicontinuous on $\X \times \A$. Recalling that the action sets
  $A(x) \equiv \A = [0, M]$ for all $x \in \X$, it follows that
  Assumption~\ref{asmp:w} holds.

  \begin{customthm}{D}\label{asmp:1}
  With positive probability, the per-period demand $D$ is greater than
  the maximum order size $M$, that is, $G_D(M) < 1$.
  \end{customthm}

\begin{prp}
  Assumption~\ref{asmp:1} implies that Assumption~\ref{asmp:ht} holds
  with $\ell = 0_L$.
\end{prp}
\begin{proof}
  Let $\gamma := 1 - G_D(M) > 0$, and let
  $\tau_L := \inf\{n \geq 1 | x_n = 0_L\}$ denote the first epoch $n$ when
  the demand $D_n$ generated a lost sale. Since the
  amount of on-hand inventory is at most $C$, and at most $M$ units
  can be ordered in a single period, it follows that
  $\P_x^\phi\{x_{\lceil C/M \rceil + 1} = 0_L\} \geq \gamma^{\lceil
    C/M \rceil + 1} > 0$ for all $\phi \in \F$ and $x \in \X$. Hence
  \begin{align*}
 &   \sum_{n=0}^\infty \mathbin{_{0_L}Q^n_\phi}\mathbf{e}(x) =
                                                                 \E_x^\phi
                                                                 \tau_L 
    = \sum_{n = 0}^\infty \P_x^\phi\{\tau_L > n\} \\
    &= 1 + \sum_{n=1}^\infty\P_x^\phi\{x_k \neq 0_L, \ k = 1, \dots,
      n\} \\
    &\leq 1 + \sum_{n=1}^\infty(1 - \gamma^{\lceil C/M
      \rceil + 1})^{\lfloor n / (\lceil C/M
      \rceil + 1)\rfloor} \\
    &\leq \frac{\lceil C/M
      \rceil + 1}{\gamma^{\lceil C/M
      \rceil + 1}} < \infty
  \end{align*}
  for all $\phi \in \F$ and $x \in \X$.
\end{proof}

\begin{thm}
  Suppose Assumption~\ref{asmp:1} holds. Then there exists a
  $\bar{\beta}$-optimal policy for the MDP defined by the HV-AG
  transformation, and every such policy is average-cost optimal for
  the original inventory control problem.
\end{thm}
\begin{proof}
  This follows from statements (ii) and (iii) of
  Corollary~\ref{cor:hv-ag-weak}.
\end{proof}

\begin{remark}
  Using the HV transformation and Corollary~\ref{cor:hv-weak}, it can
  be shown that, when Assumption~\ref{asmp:1} holds, the problem of
  minimizing the total cost incurred before the first lost sale can
  also be reduced to a discounted MDP.
\end{remark}

\paragraph{Acknowledgement} Research of the first author was
partially supported by NSF grant CMMI-1636193.

\bibliographystyle{plain}
\bibliography{orl-reduction-refs}

\section*{Appendix}

Let $\mathbb{S}$ be a metric space endowed with its Borel
$\sigma$-algebra $\B(\mathbb{S})$. A sequence $\{\nu_n\}_{n=0}^\infty$
of finite measures on $(\mathbb{S}, \B(\mathbb{S}))$ \emph{converges
  weakly} to a measure $\nu$ if, for every bounded continuous function
$f:\mathbb{S} \rightarrow \R$,
\begin{displaymath}
\lim_{n\rightarrow\infty}\int_\mathbb{S} f(x)~\nu_n(dx) =
\int_\mathbb{S} f(x)~\nu(dx).
\end{displaymath}

\begin{lem}[Dominated Convergence]\label{lem:dct}
  Let $g:\mathbb{S} \rightarrow [0, \infty)$ be a continuous function, and let
  $\{\nu_n\}_{n=0}^\infty$ be a sequence of finite measures on
  $(\mathbb{S}, \B(\mathbb{S}))$ that converges weakly to a measure $\nu$. Suppose
  there exists a continuous function $h$ on $\mathbb{S}$ such that $g \leq h$
  and
  \begin{equation}
    \label{eq:1}
    \lim_{n\rightarrow\infty}\int_\mathbb{S} h(x)~\nu_n(dx) = \int_\mathbb{S} h(x)~\nu(dx) < \infty.
  \end{equation}
  Then
  \begin{equation}
    \label{eq:2}
    \lim_{n\rightarrow\infty}\int_\mathbb{S} g(x)~\nu_n(dx) = \int_\mathbb{S} g(x)~\nu(dx).
  \end{equation}
\end{lem}
\begin{proof}
  According to \cite[Theorem~1.1]{feinberg_fatous_2014}, if
  $f:\mathbb{S} \rightarrow [0, \infty)$ is continuous, then
  \begin{align}\label{eq:fatou}
    \int_\mathbb{S} f(x)~\nu(dx) \leq \liminf_{n\rightarrow\infty}\int_\mathbb{S} f(x)~\nu_n(dx).
  \end{align}
  The equality (\ref{eq:2}) then follows by applying (\ref{eq:1}) and
  (\ref{eq:fatou}) to the nonnegative continuous functions $h - g$ and
  $h + g$.
\end{proof}

\end{document}